\documentclass[a4,12pt,epsfig]{article}
\usepackage{graphicx}
\usepackage{subfig}
\usepackage{amsmath}
\usepackage{amssymb}
\usepackage{amsfonts}
\usepackage{euscript}
\begin{document}
\bibliographystyle{plain}
\newcommand{\bea}{\begin{eqnarray}}
\newcommand{\eea}{\end{eqnarray}}
\newcommand{\bfmN}{{\mbox{\boldmath{$N$}}}}
\newcommand{\bfmx}{{\mbox{\boldmath{$x$}}}}
\newcommand{\bfmv}{{\mbox{\boldmath{$v$}}}}
\newcommand{\se}{\setcounter{equation}{0}}
\newtheorem{corollary}{Corollary}[section]
\newtheorem{example}{Example}[section]
\newtheorem{definition}{Definition}[section]
\newtheorem{theorem}{Theorem}[section]
\newtheorem{proposition}{Proposition}[section]
\newtheorem{lemma}{Lemma}[section]
\newtheorem{remark}{Remark}[section]
\newtheorem{result}{Result}[section]
\newcommand{\vtwo}{\vskip 4ex}
\newcommand{\vthree}{\vskip 6ex}
\newcommand{\vfour}{\vspace*{8ex}}
\newcommand{\hone}{\mbox{\hspace{1em}}}
\newcommand{\hon}{\mbox{\hspace{1em}}}
\newcommand{\htwo}{\mbox{\hspace{2em}}}
\newcommand{\hthree}{\mbox{\hspace{3em}}}
\newcommand{\hfour}{\mbox{\hspace{4em}}}
\newcommand{\von}{\vskip 1ex}
\newcommand{\vone}{\vskip 2ex}
\newcommand{\n}{\mathfrak{n} }
\newcommand{\m}{\mathfrak{m} }
\newcommand{\q}{\mathfrak{q} }
\newcommand{\aF}{\mathfrak{a} }

\newcommand{\kl}{\mathcal{K}}
\newcommand{\p}{\mathcal{P}}
\newcommand{\Lt}{\mathcal{L}}
\newcommand{\bv}{{\mbox{\boldmath{$v$}}}}
\newcommand{\bc}{{\mbox{\boldmath{$c$}}}}
\newcommand{\bx}{{\mbox{\boldmath{$x$}}}}
\newcommand{\br}{{\mbox{\boldmath{$r$}}}}
\newcommand{\bs}{{\mbox{\boldmath{$s$}}}}
\newcommand{\bb}{{\mbox{\boldmath{$b$}}}}
\newcommand{\ba}{{\mbox{\boldmath{$a$}}}}
\newcommand{\bn}{{\mbox{\boldmath{$n$}}}}
\newcommand{\bp}{{\mbox{\boldmath{$p$}}}}
\newcommand{\by}{{\mbox{\boldmath{$y$}}}}
\newcommand{\bz}{{\mbox{\boldmath{$z$}}}}
\newcommand{\be}{{\mbox{\boldmath{$e$}}}}
\newcommand{\proof}{\noindent {\sc Proof :} \par }
\newcommand{\bP}{{\mbox{\boldmath{$P$}}}}

\newcommand{\M}{\mathcal{M}}
\newcommand{\R}{\mathbb{R}}
\newcommand{\Q}{\mathbb{Q}}
\newcommand{\Z}{\mathbb{Z}}
\newcommand{\N}{\mathbb{N}}
\newcommand{\C}{\mathbb{C}}
\newcommand{\xar}{\longrightarrow}
\newcommand{\ov}{\overline}
 \newcommand{\rt}{\rightarrow}
 \newcommand{\om}{\omega}
 \newcommand{\wh}{\widehat }
 \newcommand{\wt}{\widetilde }
 \newcommand{\g}{\Gamma}
 \newcommand{\lm}{\lambda}

\newcommand{\eN}{\EuScript{N}}
\newcommand{\ncom}{\newcommand}
\newcommand{\norm}{\|\;\;\|}
\newcommand{\inp}[2]{\langle{#1},\,{#2} \rangle}
\newcommand{\nrm}[1]{\parallel {#1} \parallel}
\newcommand{\nrms}[1]{\parallel {#1} \parallel^2}
\title{On the Convergence of Quasilinear Viscous Approximations Using Compensated Compactness}% and Their Convergence Rate}
\author{ Ramesh Mondal\footnote{ramesh@math.iitb.ac.in}\,\, and S. Sivaji Ganesh\footnote{siva@math.iitb.ac.in}}
\maketitle{}
\begin{abstract}
Method of compensated compactness is used to show that the almost everywhere limit of quasilinear viscous approximations is the unique entropy
solution (in the sense of {\it Bardos et.al}\cite{MR542510}) of the corresponding scalar conservation laws in a bounded domain in $\mathbb{R}^{d}$, where the viscous term is
of the form $\varepsilon div\left(B(u^{\varepsilon})\nabla u^{\varepsilon}\right)$.
 \end{abstract}
 \section{Introduction}
 Let $\Omega$ be a bounded domain in $\mathbb{R}^{d}$ with smooth boundary $\partial \Omega$. For $T >0$, denote $\Omega_{T}:= \Omega\times(0,T)$. 
 We write the initial boundary value problem $\left(\mbox{IBVP}\right)$ for scalar conservation laws given by
 \begin{subequations}\label{ivp.cl}
\begin{eqnarray}
  u_t + \nabla\cdot f(u) =0& \mbox{in }\Omega_T,\label{ivp.cl.a}\\
u(x,t)= 0&\mbox{on}\,\,\partial \Omega\times(0,T),\label{ivp.cl.b}\\
  u(x,0) = u_0(x)& x\in \Omega.\label{ivp.cl.c}
  \end{eqnarray}
\end{subequations}
where $f=(f_{1},f_{2},\cdots,f_{d})$ is the flux function and $u_{0}$ is the initial condition.\\
Denote by  $u_{0\varepsilon}$, the regularizations of the initial condition $u_{0}$ of IBVP \eqref{ivp.cl}, using the standard
sequence of mollifiers $\rho_{\varepsilon}$ defined on $\R^d$. It is given by
\begin{eqnarray*}\label{regularized.eqn2}
 u_{0\varepsilon} := u_{0}\ast\rho_{\varepsilon}.\nonumber
\end{eqnarray*}
Consider the IBVP for the viscosity problem 
\begin{subequations}\label{regularized.IBVP}
\begin{eqnarray}
 u^\varepsilon_{t} + \nabla \cdot f(u^{\varepsilon}) = \varepsilon\,\nabla\cdot\left(B(u^\varepsilon)\,\nabla u^\varepsilon\right)
 &\mbox{in }\Omega_{T},\label{regularized.IBVP.a} \\
    u^\varepsilon(x,t)= 0&\,\,\,\,\mbox{on}\,\, \partial \Omega\times(0,T),\label{regularized.IBVP.b}\\
u^{\varepsilon}(x,0) = u_{0\varepsilon}(x)& x\in \Omega,\label{regularized.IBVP.c}
\end{eqnarray}
\end{subequations}
indexed by $\varepsilon>0$. The aim of this article is to prove that the {\it a.e.} limit of sequence of solutions 
$\left(u^{\varepsilon}\right)$ to \eqref{regularized.IBVP}(called quasilinear viscous approximations) is the unique entropy solution for IBVP \eqref{ivp.cl}.\\
Let us write the hypothesis on $f,\,B,\,$ and $u_{0}$. \\
\noindent{\bf Hypothesis D:}
\begin{enumerate}
 \item Let $f\in \left(C^4(\R)\right)^d$, $f^\prime\in \left(L^\infty(\R)\right)^d$, and denote 
 $$\|f^\prime\|_{\left(L^\infty(\R)\right)^d}:=\max_{1\leq j\leq d}\,\sup_{y\in\R}|f^\prime_j(y)|.$$
 \item Let $B\in C^3(\R)\cap L^\infty(\R)$, and there exists an $r>0$ such that $B\geq r$.
 \item  Let the space $L^{\infty}_c({\Omega})$ consisting of those elements of $L^{\infty}({\Omega})$ whose essential 
 support is a compact subset of $\Omega$. Let $u_{0}$ be in $H^{1}(\Omega)\cap L^{\infty}_{c}(\Omega)$ and we denote 
 $I:=[-\|u_0\|_{\infty},\|u_0\|_{\infty}]$.
\end{enumerate}
In this context, we have the following main result. 
\begin{theorem}\label{paper2.compensatedcompactness.theorem2}
 Let $f,\,B,\,u_{0}$ satisfy Hypothesis D. Then the {\it a.e.} limit of the quasilinear viscous approximations 
 $\left(u^{\varepsilon}\right)$ is the unique entropy solution of IBVP \eqref{ivp.cl} in the sense of {\it Bardos et.al}\cite{MR542510}.
\end{theorem}
In \cite{Ramesh}, we prove BV estimates and  as a consequence of that, we have the existence of an almost everywhere convergent 
subsequence of quasilinear viscous approximations $\left(u^{\varepsilon}\right)$. But in this article, we use method of compensated
compactness to show the existence of an almost everywhere convergent subsequence of quasilinear viscous approximations $\left(u^{\varepsilon}\right)$.

The plan of the paper is the following. In section 2, we prove Existence, uniqueness, maximum principle and derivative estimates of 
the quasilinear viscous approximations and In section 3, we prove compactness of sequence of quasilinear viscous approximations
$\left(u^{\varepsilon}\right)$ and Theorem \ref{paper2.compensatedcompactness.theorem2}.
\section{Existence, uniqueness, maximum principle and derivative estimates}
We now want to use the following higher regularity result from \cite[p.18]{Ramesh} of the generalized viscosity problem
 \begin{subequations}\label{ibvp.parab}
\begin{eqnarray}
 u_{t} + \nabla \cdot f(u) = \varepsilon\,\nabla\cdot\left(B(u)\,\nabla u\right)&\mbox{in }\Omega_{T},\label{ibvp.parab.a} \\
    u(x,t)= 0&\,\,\,\,\mbox{on}\,\, \partial \Omega\times(0,T),\label{ibvp.parab.b}\\
u(x,0) = u_0(x)& x\in \Omega,\label{ibvp.parab.c}
\end{eqnarray}
\end{subequations}
\begin{theorem}[higher regularity]\label{chapHR85thm5}
Let $0 <\beta < 1$, $f\in\left(C^{4}(\mathbb{R})\right)^{d}$, $B\in C^{3}(\mathbb{R})\,\mbox{ with}\,\, B\geq r> 0$. 
Let $u_{0}\in C^{4+\beta}(\overline{\Omega})$ having compact essential support in $\Omega$. 
Then the solutions of the IBVP \eqref{ibvp.parab} belong to the space $C^{4+\beta,\frac{4 + \beta}{2}}(\overline{\Omega_{T}})$. Further,  $u_{tt}^{\varepsilon}\in C(\overline{\Omega_{T}})$.
\end{theorem}
Since $u_{0}\in L^{\infty}_{c}(\Omega)$, the function $u_{0\varepsilon}$ belongs to the space $C^{\infty}(\overline{\Omega})$ and also has compact support in $\Omega$ for sufficiently small $\varepsilon$. 
As a consequence the initial-boundary data of the regularized generalized viscosity problem \eqref{regularized.IBVP} satisfies 
compatibility conditions of orders $0,1,2$ which are required to apply the higher regularity result Theorem \ref{chapHR85thm5} to conclude 
the following result.
\begin{theorem}[Existence and Uniqueness of Solutions]\label{chapHR85thm5ABCD}
Let $f,\,B$ and $u_{0}$ be as in the Hypothesis D. Then the solutions of the IBVP \eqref{regularized.IBVP} belong to the space $C^{4+\beta,\frac{4 + \beta}{2}}(\overline{\Omega_{T}})$. Further,  $u_{tt}^{\varepsilon}\in C(\overline{\Omega_{T}})$.
\end{theorem}
We now recall the maximum principle of generalized viscosity problem \eqref{ibvp.parab} from \cite[p.12]{Ramesh}, {\it i.e.,}
\begin{theorem}[Maximum principle]\label{chap3thm1}
Let $f:\mathbb{R}\to\mathbb{R}^{d}$ be a $C^{1}$ function and $u_{0}\in L^{\infty}(\Omega)$. Then any solution $u$ of generalized viscosity problem \eqref{ibvp.parab} in $W(0,T)$  satisfies the bound
\begin{equation}\label{eqnchap303}
||u^{\varepsilon}(\cdot,t)||_{L^{\infty}(\Omega)}\hspace{0.1cm}\leq\hspace{0.1cm}||u_{0}||_{L^{\infty}(\Omega)}\hspace{0.1cm}a.e.\,\,t\in(0,T).
\end{equation}
\end{theorem}
Applying Theorem \ref{chap3thm1} to regularized generalized viscosity problem \eqref{regularized.IBVP} and using $\|u_{0\varepsilon}\|_{L^{\infty}(\Omega)}
\leq \|u_{0}\|_{L^{\infty}(\Omega)}$, we get the following maximum principle
\begin{theorem}[Maximum principle]\label{regularized.chap3thm1}
Let $f,\,B,\,$ and $u_{0}$ be as in Hypothesis D. Then any solution $u$ of generalized viscosity problem \eqref{regularized.IBVP}  satisfies the bound
\begin{equation}\label{regularized.eqnchap303}
||u^{\varepsilon}(\cdot,t)||_{L^{\infty}(\Omega)}\hspace{0.1cm}\leq\hspace{0.1cm}||u_{0}||_{L^{\infty}(\Omega)}\hspace{0.1cm}a.e.\,\,t\in(0,T).
\end{equation}
\end{theorem}
Applying Theorem 4.2 from \cite[p.30]{Ramesh} to quasilinear viscous approximations $\left(u^{\varepsilon}\right)$ as asserted in
Theorem \ref{chapHR85thm5ABCD}, we obtain
\begin{theorem}\label{Compactness.lemma.1}
Let $f,\,\,B,\,\,u_{0}$ satisfy Hypothesis D. Let $u^{\varepsilon}$ be the unique solution to generalized 
viscosity problem \eqref{regularized.IBVP}. Then 
 \begin{eqnarray}\label{uniformnot.compactness.eqn1a}
 \displaystyle\sum_{j=1}^{d} \hspace{0.1cm}\left(\sqrt{\varepsilon}\Big\| \frac{\partial u^{\varepsilon}}{\partial x_{j}}\Big\|_{L^{2}(\Omega_{T})}\right)^{2} \leq\frac{1}{2r}\|u_{0\varepsilon}\|^{2}_{L^{2}(\Omega)}\leq\frac{1}{2r}
 \|u_{0}\|^{2}_{L^{\infty}(\Omega)}\,\,\mbox{Vol}(\Omega).
 \end{eqnarray}
\end{theorem}
\section{Compactness of quasilinear viscous approximations}
In this section, we want to show the existence of a subsequence $\left\{u^{\varepsilon_{k}}\right\}_{k=1}^{\infty}$ of sequence of solutions $\left\{u_{\varepsilon}\right\}_{\varepsilon \geq 0}$ to generalized viscosity problem \eqref{ibvp.parab} and a function $u$ in $L^{\infty}(\Omega_{T})$ such that for {\it a.e.} $(x,t)\in \Omega_{T}$, we have
\begin{equation*}
u^{\varepsilon_{k}}(x,t)\to u(x,t)\hspace{0.2cm}{\it a.e.}\,\,(x,t)\in\Omega_{T}\,\,\,\mbox{as}\,\,k\to\infty.
\end{equation*}
The following result shows that the quasilinear viscous approximations $\left(u^{\varepsilon}\right)$ satisfies compact entropy
productions, {\it i.e.,}
\begin{theorem}\label{chap9thm2}
{\rm Assume \textbf{Hypothesis D} and let $\left\{u^{\varepsilon}\right\}$ be as in Theorem \ref{chapHR85thm5ABCD}.
Then 
 \begin{equation}\label{chap9eqn2}
  \frac{\partial \eta(u^{\varepsilon})}{\partial t} + \displaystyle\sum_{j=1}^{d}\frac{\partial q_{j}(u^{\varepsilon})}{\partial x_{j}}\hspace{0.2cm}\subset\hspace{0.1cm}\mbox{compact set in }\hspace{0.2cm} H^{-1}(\Omega_{T})
 \end{equation}
for every $C^{2}(\mathbb{R})$ entropy-entropy flux pair $(\eta, q)$ of conservation laws \eqref{ivp.cl.a}.}
\end{theorem}
The following result is used to prove \eqref{chap9eqn2}.
\begin{lemma}{\rm\cite[p.514]{MR2169977}}\label{chap9lem2}
 {\rm Let $\Omega$ be an open subset of $\mathbb{R}^{d}$ and $\left\{\phi_{n}\right\}$ be a bounded sequence in 
 $W^{-1,p}(\Omega)$, for some $p > 2$. Further, let $\phi_{n}= \xi_{n} + \psi_{n}$, where $\left\{\xi_{n}\right\}$ 
 lies in a compact set of $H^{-1}(\Omega)$, while $\left\{\psi_{n}\right\}$ lies in a bounded set of the space of measures
 $M(\Omega)$. Then $\left\{\phi_{n}\right\}$ lies in a compact set of $H^{-1}(\Omega)$.} 
\end{lemma}
\textbf{Proof of Theorem \ref{chap9thm2}:} Let $\eta:\mathbb{R}\to \mathbb{R}$ be a convex, $C^{2}(\mathbb{R})$ entropy. Then for $j=1,2,\cdots,d$, there exists $C^{2}(\mathbb{R})$ functions $q_{j}:\mathbb{R}\to\mathbb{R}$ such that
 \begin{equation}\label{compactnessequation1}
  \eta^{\prime}f_{j}^{\prime}= q_{j}^{\prime}.
 \end{equation}
 Multiplying both sides the equation \eqref{ibvp.parab.a} by $\eta^{\prime}(u^{\varepsilon})$ and using chain rule, we get
 \begin{equation}\label{compactnessequation2}
 \frac{\partial \eta(u^{\varepsilon})}{\partial t} + \displaystyle\sum_{j=1}^{d}\frac{\partial q_{j}(u^{\varepsilon})}{\partial x_{j}} = \varepsilon\hspace{0.1cm}\displaystyle\sum_{j=1}^{d}\frac{\partial }{\partial x_{j}}\left(B(u^{\varepsilon})\frac{\partial u^{\varepsilon}}{\partial x_{j}}\right)\hspace{0.1cm}\eta^{\prime}(u^{\varepsilon}).
\end{equation}
From equation \eqref{compactnessequation2}, we get
\begin{equation}\label{chap9eqn4}
 \frac{\partial \eta(u^{\varepsilon})}{\partial t} + \displaystyle\sum_{j=1}^{d}\frac{\partial q_{j}(u^{\varepsilon})}{\partial x_{j}} = \varepsilon\hspace{0.1cm}\displaystyle\sum_{j=1}^{d}\frac{\partial }{\partial x_{j}}\left(B(u^{\varepsilon})\frac{\partial \eta(u^{\varepsilon})}{\partial x_{j}}\right) - \varepsilon \displaystyle\sum_{j=1}^{d}B(u^{\varepsilon})\hspace{0.1cm}\left(\frac{\partial u}{\partial x_{j}}\right)^{2}\hspace{0.1cm}\eta^{\prime\prime}(u^{\varepsilon})
\end{equation}
By appealing to Lemma \ref{chap9lem2}, \eqref{chap9eqn2} will be proved if we prove 
\begin{equation}\label{compactnessequation3}
 \varepsilon\hspace{0.1cm}\displaystyle\sum_{j=1}^{d}\frac{\partial }{\partial x_{j}}\left(B(u^{\varepsilon})\frac{\partial \eta(u^{\varepsilon})}{\partial x_{j}}\right)\to 0\hspace{0.1cm}\mbox{in}\hspace{0.1cm}H^{-1}(\Omega_{T})
\end{equation}
and 
\begin{equation}\label{compactnessequation4}
 - \varepsilon \displaystyle\sum_{j=1}^{d}B(u^{\varepsilon})\hspace{0.1cm}\left(\frac{\partial u}{\partial x_{j}}\right)^{2}\hspace{0.1cm}\eta^{\prime\prime}(u^{\varepsilon})\hspace{0.2cm}\mbox{is bounded in the space of measures}\hspace{0.1cm}M(\Omega_{T}).
\end{equation}
Denote
$$I := [-\|u_{0}\|_{L^{\infty}(\mathbb{R})},\|u_{0}\|_{L^{\infty}(\mathbb{R})}].$$
Firstly, we prove \eqref{compactnessequation3}. Note that
\begin{equation}\label{chap9eqn5}
 \Big\|\varepsilon\hspace{0.1cm}\displaystyle\sum_{j=1}^{d}\frac{\partial }{\partial x_{j}}\left(B(u^{\varepsilon})\frac{\partial \eta(u^{\varepsilon})}{\partial x_{j}}\right)\Big\|_{H^{-1}(\Omega_{T})} = \displaystyle\sup_{\substack{\phi\in H^{1}_{0}(\Omega_{T}),\\ \|\phi\|_{H^{1}_{0}(\Omega_{T})}\leq 1}}\left|\int_{0}^{T}\int_{\Omega}\varepsilon\hspace{0.1cm}\displaystyle\sum_{j=1}^{d}\frac{\partial }{\partial x_{j}}\left(B(u^{\varepsilon})\frac{\partial \eta(u^{\varepsilon})}{\partial x_{j}}\right)\hspace{0.1cm}\phi\hspace{0.1cm}dx\hspace{0.1cm}dt\right|. 
\end{equation}
Using integration by parts formula and  $\|\phi\|_{H^{1}_{0}(\Omega_{T})}\leq 1$ in \eqref{chap9eqn5}, we arrive at 
\begin{equation}\label{chap9eqn6}
 \left|-\int_{0}^{T}\int_{\Omega}\varepsilon\hspace{0.1cm}\displaystyle\sum_{j=1}^{d}B(u^{\varepsilon})\eta^{\prime}(u^{\varepsilon})\frac{\partial u^{\varepsilon}}{\partial x_{j}}\hspace{0.1cm}\frac{\partial \phi}{\partial x_{j}}\hspace{0.1cm}dx\hspace{0.1cm}dt\right| \leq \varepsilon \|B\|_{L^{\infty}(\mathbb{R})}\|\eta^{'}\|_{L^{\infty}(I)}\|\nabla u^{\varepsilon}\|_{\left(L^{2}(\Omega_{T})\right)^{d}}.
\end{equation}
From  Theorem \ref{Compactness.lemma.1}, we have
\begin{equation}\label{Comp.chap9eqn26}
  \displaystyle\sum_{j=1}^{d} \hspace{0.1cm}\left(\sqrt{\varepsilon}\Big\| \frac{\partial u^{\varepsilon}}{\partial x_{j}}\Big\|_{L^{2}(\Omega_{T})}\right)^{2} \leq\frac{1}{2r}\|u_{0\varepsilon}\|^{2}_{L^{2}(\Omega)}\leq
  \frac{1}{2r}\|u_{0}\|^{2}_{L^{\infty}(\Omega)}\mbox{Vol}(\Omega).
\end{equation}
Using \eqref{Comp.chap9eqn26} in \eqref{chap9eqn6} and letting $\varepsilon\to 0$ in \eqref{chap9eqn6}, we have \eqref{compactnessequation3}.\\
Secondly, We want to prove \eqref{compactnessequation4}.  We have $- \varepsilon \displaystyle\sum_{j=1}^{d}B(u^{\varepsilon})\hspace{0.1cm}\left(\frac{\partial u}{\partial x_{j}}\right)^{2}\hspace{0.1cm}\eta^{''}(u^{\varepsilon})\in L^{1}(\Omega_{T})$. We know that $L^{1}(\Omega_{T})$ is continuously imbedded in $\left(L^{\infty}(\Omega_{T})\right)^{\ast}$. Therefore, we have 
\begin{eqnarray}\label{chap9eqn11}
 \Big\| -\varepsilon \displaystyle\sum_{j=1}^{d}B(u^{\varepsilon})\hspace{0.1cm}\left(\frac{\partial u}{\partial x_{j}}\right)^{2}\hspace{0.1cm}\eta^{''}(u^{\varepsilon})\Big\|_{M(\Omega_{T})} &\leq& \Big\|-\varepsilon \displaystyle\sum_{j=1}^{d}B(u^{\varepsilon})\hspace{0.1cm}\left(\frac{\partial u}{\partial x_{j}}\right)^{2}\hspace{0.1cm}\eta^{''}(u^{\varepsilon})\Big\|_{L^{1}(\Omega_{T})},\nonumber\\
 &\leq& \varepsilon \|B\|_{L^{\infty}(\mathbb{R})}\hspace{0.1cm}\|\eta^{''}\|_{L^{\infty}(I)}\|\nabla u\|_{\left(L^{2}(\Omega_{T})\right)^{d}}.
\end{eqnarray}
Using inequality \eqref{Comp.chap9eqn26} in \eqref{chap9eqn11}, we get
\begin{equation}\label{compactnessequation9}
 \Big\| -\varepsilon \displaystyle\sum_{j=1}^{d}B(u^{\varepsilon})\hspace{0.1cm}\left(\frac{\partial u}{\partial x_{j}}\right)^{2}\hspace{0.1cm}\eta^{''}(u^{\varepsilon})\Big\|_{M(\Omega_{T})} \leq C^{'}\|B\|_{L^{\infty}(\mathbb{R})}\|\eta^{''}\|_{L^{\infty}(I)},
\end{equation}
where $C^{'}$ is independent of $\varepsilon$. Therefore we have obtained \eqref{compactnessequation4}.
Using \eqref{compactnessequation3}, \eqref{compactnessequation4} and in view of Lemma \ref{chap9lem2}, we have \eqref{chap9eqn2}.
\\

For space dimension $d=1$, the extraction of an {\it a.e.} convergent subsequence is obtained by proving the following result.

\begin{theorem}\label{chap9thm3}
{\rm Assume \textbf{Hypothesis D} and let $\left\{u^{\varepsilon}\right\}$ be sequence of solutions to generalized viscosity problem \eqref{ibvp.parab} such that 
 \begin{equation}\label{chap9eqn13}
  \frac{\partial \eta(u^{\varepsilon})}{\partial t} + \frac{\partial q(u^{\varepsilon})}{\partial x}\hspace{0.2cm}\subset\hspace{0.1cm}\mbox{compact set in }\hspace{0.2cm} H^{-1}(\Omega_{T}).
 \end{equation}
for every $C^{2}(\mathbb{R})$ entropy-entropy flux pair $(\eta, q)$ of scalar conservation laws \eqref{ibvp.parab.a} in one space dimension. Then there is a subsequence of $u^{\varepsilon}$ such that the following 
convergence in $L^{\infty}(\Omega_{T})-\mbox{weak}^{\ast}$
\begin{equation*}
 u^{\varepsilon}\rightharpoonup u,\hspace{0.3cm}f(u^{\varepsilon})\rightharpoonup f(u),\hspace{0.2cm}\mbox{as}\hspace{0.1cm}\varepsilon\to 0.
\end{equation*}
holds. Further, if the set of $u$ with with $f^{''}(u)\neq 0$ is dense in $\mathbb{R}$, then $\left\{u^{\varepsilon}\right\}$ converges almost everywhere to $u$ in $\Omega_{T}$.}
\end{theorem}
The following results are used to prove proving Theorem \ref{chap9thm3}. 
\begin{theorem}{\rm \cite[p.147]{MR584398}}\label{chap9thm1}
\begin{enumerate}
\item {\rm Let $K\subset \mathbb{R}^{p}$ be bounded and $\Omega\subset \mathbb{R}^{d}$ be an open set. Let $u_{n}:\Omega\to\mathbb{R}^{p}$ be such that $u_{n}\in K$ {\it a.e.}. Then there exists a subsequence $\left\{u_{m}\right\}$ and a family of probability measures $\left\{\nu_{x}\right\}_{x\in\Omega}$ $\left(\mbox{depending measurably on x}\right)$ with $\mbox{supp}\hspace{0.1cm}\nu_{x}\subset \overline{K}$ such that if $F$ is continuous function on $\mathbb{R}^{p}$ and 
\begin{equation*}
 \overline{f} = <\nu_{x}, F(\lambda)>\,\, {\it a.e.}
\end{equation*}
then 
\begin{equation*}
 F(u_{m})\rightharpoonup \overline{f}(x)\hspace{0.2cm}\mbox{in}\hspace{0.2cm}L^{\infty}(\Omega)-\mbox{weak}^{\ast}
\end{equation*}
}
\item {\rm Conversely, let $\left\{\nu_{x}\right\}_{x\in\Omega}$ be a family of  probability measures with support in $ \overline{K}$. Then there exists a sequence  $\left\{u_{n}\right\}$, where $u_{n}:\Omega\to\mathbb{R}^{p}$ and $u_{n}\in K$ {\it a.e.}, such that for all continuous functions on $\mathbb{R}^{p}$, we have
\begin{equation*}
 F(u_{n})\rightharpoonup \overline{f}(x)=<\nu_{x}, F(\lambda)>\hspace{0.2cm}\mbox{in}\hspace{0.2cm}L^{\infty}(\Omega)-\mbox{weak}^{\ast}
\end{equation*}
}
\end{enumerate}
\end{theorem}
We give definition of Young measures.
\begin{definition}
 {\rm The family of probability measure $\left\{\nu_{x}\right\}_{x\in \Omega}$ that we get from Theorem \ref{chap9thm1} is called Young measures associated with the sequence $\left\{u_{n}\right\}_{n=1}^{\infty}$.}
\end{definition}
\begin{lemma}(\textbf{Div-curl lemma}){\rm \cite[p.90]{MR2582099},\cite[p.513]{MR2169977}}\label{chap9lem1}\\
 {\em Let $\Omega\subset\mathbb{R}^{d}$ be an open set and $G_{n}$ and $H_{n}$ be two sequences of vector fields in $L^{2}(\Omega;\mathbb{R}^{d})$ converging weakly to limits $\overline{G}$ and $\overline{H}$ respectively as $n\to \infty$. Assume both $\left\{\mbox{div}\hspace{0.1cm}G_{n}\right\}$ and $\left\{\mbox{curl}\,H_{n}\right\}$ lie in a compact subset of $H^{-1}_{loc}(\Omega)$. Then we have the following convergence in the sense of distributions as $n\to \infty$ 
 \begin{equation*}
  G_{n}.H_{n}\to \overline{G}.\overline{H}.
 \end{equation*}
}
\end{lemma}
\vspace{0.2cm}
For clarity, we repeat the following proof from \cite[p.518]{MR2169977}.\\
\textbf{Proof of Theorem \ref{chap9thm3}:}
  Since $\|u^{\varepsilon}\|_{L^{\infty}(\Omega_{T})}\leq \|u_{0}\|_{L^{\infty}(\Omega_{T})}$, by Banach-Alaoglu theorem there exists a subsequence, still denoted by $\left\{u^{\varepsilon}\right\}$ such that
  \begin{center}
   $u^{\varepsilon}\rightharpoonup u,\hspace{0.3cm}\mbox{as}\hspace{0.1cm}\varepsilon\to 0$
  \end{center}
 in $L^{\infty}(\Omega_{T})-\mbox{weak}^{\ast}$. Applying Theorem \ref{chap9thm1}, we get the existence of a family of Young measures $\left\{\nu_{x,t}\right\}_{\left(x,t\right)\in \Omega_{T}}$ such that 
 \begin{equation*}
  f(u_{s})\rightharpoonup \overline{f}(x,t)\hspace{0.2cm}\mbox{in}\hspace{0.2cm}L^{\infty}(\Omega_{T})-\mbox{weak}^{\ast},
 \end{equation*}
where 
\begin{equation}\label{chap9eqn14}
 \overline{f}(x,t) = \int_{\mathbb{R}}f(\lambda)\hspace{0.1cm}d\nu_{x,t}(\lambda) = \langle \nu_{x,t},f \rangle.
\end{equation}
 We want to show $\langle \nu_{x,t},f \rangle = f(u)$, that is, $\nu_{x,t}$ reduces to dirac mass when there is no interval on which $f^{''}$ is constant.  Define
 $$g(\lambda):=\int_{0}^{\lambda}\left(f^{'}(s)\right)^{2}\hspace{0.1cm}ds. $$
 Let  $\left(\lambda,f(\lambda)\right)$ and $\left(f(\lambda),g(\lambda)\right)$ be two entropy-entropy flux pair. Using Div-curl lemma with sequences $(u^{\varepsilon},f(u^{\varepsilon}))$ and $(f(u^{\varepsilon}),g(u^{\varepsilon}))$, we obtain
\begin{equation}\label{chap9eqn15}
 (u^{\varepsilon},f(u^{\varepsilon})).(f(u^{\varepsilon}),g(u^{\varepsilon}))\rightharpoonup \langle\nu_{x,t},u\rangle \langle\nu_{x,t},g\rangle - \langle\nu_{x,t},f\rangle \langle\nu_{x,t},f\rangle
\end{equation}
in $L^{\infty}(\Omega_{T})-\mbox{weak}^{\ast}$. 
Observe that
\begin{equation}\label{chap9eqn17}
 \langle\nu_{x,t},\lambda\rangle \langle\nu_{x,t},g\rangle - \langle\nu_{x,t},f\rangle \langle\nu_{x,t},f\rangle = \langle \nu_{x,t}, \lambda g-f^{2}\rangle
\end{equation}
Using Schwartz inequality, we get
\begin{equation}\label{chap9eqn16}
 \left(f(\lambda)-f(u)\right)^{2}\leq \left(\lambda - u\right)\left(g(\lambda)-g(u)\right),
\end{equation}
Since probability measures are positive, from inequality \eqref{chap9eqn16}, we have
\begin{equation}\label{chap9eqn18}
 \langle\nu_{x,t}, \left(f(\lambda)-f(u)\right)^{2}- \left(\lambda - u\right)\left(g(\lambda)-g(u)\right)\rangle \leq 0.
\end{equation}
Using \eqref{chap9eqn17} in \eqref{chap9eqn18}, we get
\begin{equation}\label{chap9eqn19}
 \left(\langle\nu_{x,t},f\rangle - f(u)\right)^{2}\leq 0
\end{equation}
Equation \eqref{chap9eqn19} shows that $\langle\nu_{x,t},f\rangle = f(u)$. Therefore we have proved that the $\nu_{x,t}$ reduces to the Dirac measure $\delta_{u(x,t)}$. 
The Schwartz's inequality is an equality on the interval where $f^{'}$ is constant with end points $\lambda,u$. 
When no such interval exists, the support of $\nu_{x,t}$ collapses to a single point and $\nu_{x,t}$ reduces to Dirac mass $\delta_{u(x,t)}$. 
This completes the proof.\\
We now show that the $\left\{\frac{\partial u^{\varepsilon}}{\partial t}\right\}$ lies in a compact set of $H^{-1}_{loc}(\Omega_{T})$.
The following result will be used in proving compactness.
\begin{lemma}\cite[p.1020]{MR542510}\label{compactness12.lem1}
 Let $v\in C^{1}(\overline{\Omega})$. Then 
 $$\displaystyle\lim_{n\to\infty}\int_{\left\{x\in\Omega\,;\,|v(x)|<\frac{1}{n}\right\}}\,\left|\nabla v\right|\,dx=0.$$
\end{lemma}
The following result follows from \cite[p.67]{MR1304494} which is useful in proving, and we omit its proof. 
\begin{lemma}\label{regularized.lem1}
  Let $u_{0}\in W^{1,2}(\Omega)\cap L^{\infty}_{c}(\Omega)$. Then $u_{0\varepsilon}$ satisfies the following bounds
  \begin{eqnarray}
    \|u_{0\varepsilon}\|_{L^{\infty}(\Omega)}\leq \|u_{0}\|_{L^{\infty}(\Omega)}\label{regularized.max.eqn1}\\[2mm]
     \|\nabla u_{0\varepsilon}\|_{\left(L^{1}(\Omega)\right)^{d}}\leq TV_{\Omega}(u_{0}) \label{regularized.max.eqn1a}
  \end{eqnarray}
There exists a constant $C> 0$ such that for all $\varepsilon > 0$, $u_{0\varepsilon}$ satisfies 
  \begin{eqnarray}\label{regularized.max.eqn1b}
   \|\Delta u_{0\varepsilon}\|_{L^{1}(\Omega)}\leq \frac{C}{\varepsilon}TV_{\Omega}(u_{0}).
  \end{eqnarray}
\end{lemma}
\begin{theorem}\label{timedervative.thm1}
 Assume Hypothesis C. Let $\left(u^{\varepsilon}\right)$ be the sequence of solutions to \eqref{regularized.IBVP}. Then 
 $\left\{\frac{\partial u^{\varepsilon}}{\partial t}\right\}$ is compact in $H^{-1}_{loc}(\Omega_{T})$.
\end{theorem}
\begin{proof}
 We prove Theorem \ref{timedervative.thm1} in two steps. In Step-1, we show that the sequence $\left\{\frac{\partial u^{\varepsilon}}{\partial t}\right\}$
 is bounded in $L^{1}(\Omega_{T})$ and in Step-2, we use Murat's Lemma \ref{chap9lem2} to show that $\left\{\frac{\partial u^{\varepsilon}}{\partial t}\right\}$
 is compact in $H^{-1}_{loc}(\Omega_{T})$.\\
 \textbf{Step-1:} We now repeat the proof of the $L^{1}(\Omega_{T})-$ estimates of the time derivatives of sequence of solutions to generalized viscosity 
 problem \eqref{ibvp.parab} from \cite{Ramesh} to show that the sequence $\left\{\frac{\partial u^{\varepsilon}}{\partial t}\right\}$ is bounded in 
 $L^{1}(\Omega_{T})$. \\
  We show the existence of constant $C_{1}> 0$ such that for every $\varepsilon>0$,
\begin{eqnarray}\label{B.BVestimate26}
 \left\|\frac{\partial u^{\varepsilon}}{\partial t}\right\|_{L^{1}(\Omega_{T})}\leq C_{1}.
\end{eqnarray}
Differentiating the equation  \eqref{regularized.IBVP} with respect to $t$, multiplying by $sg_{n}\left(\frac{\partial u^{\varepsilon}}{\partial t}\right)$ and integrating over $\Omega$, we get
%  \begin{subequations}
\begin{eqnarray}
 \int_{\Omega} u^\varepsilon_{tt}\,sg_{n}(u^\varepsilon_{t})\,dx + \displaystyle\sum_{j=1}^{d}\int_{\Omega}\Big[\frac{\partial}{\partial x_{j}}\left(f_{j}^{\prime}(u^{\varepsilon})\,\frac{\partial u^{\varepsilon}}{\partial t}\right)\Big]sg_{n}\left(\frac{\partial u^{\varepsilon}}{\partial t}\right)\,dx \hspace*{1in}\nonumber\\=\varepsilon\displaystyle\sum_{j=1}^{d}\int_{\Omega}\,sg_{n}\left(\frac{\partial u^{\varepsilon}}{\partial t}\right)\,\frac{\partial}{\partial x_{j}}\Big(B^{'}(u^{\varepsilon})\,\frac{\partial u^{\varepsilon}}{\partial t}\,\frac{\partial u^{\varepsilon}}{\partial x_{j}} + B(u^{\varepsilon})\,\frac{\partial}{\partial x_{j}}\left(\frac{\partial u^{\varepsilon}}{\partial t}\right)\Big)\,dx.\label{compactness12.eqn3}
 \end{eqnarray}
% \end{subequations}
Using integration by parts in \eqref{compactness12.eqn3} and using $sg_{n}\left(\frac{\partial u^{\varepsilon}}{\partial t}\right)=0\,\,\mbox{on}\,\,\partial\Omega\times(0,T)$, we have
\begin{eqnarray}
  \int_{\Omega} u^\varepsilon_{tt}\,sg_{n}(u^\varepsilon_{t})\,dx = \displaystyle\sum_{j=1}^{d}\int_{\Omega}f_{j}^{\prime}(u^{\varepsilon})\,\frac{\partial u^{\varepsilon}}{\partial t}\,sg^{'}_{n}\left(\frac{\partial u^{\varepsilon}}{\partial t}\right)\,\frac{\partial}{\partial x_{j}}\left(\frac{\partial u^{\varepsilon}}{\partial t}\right)\,dx \hspace*{1in}\nonumber\\
-\varepsilon\displaystyle\sum_{j=1}^{d}\int_{\Omega}B^{'}(u^{\varepsilon})\,\frac{\partial u^{\varepsilon}}{\partial t}\,\frac{\partial u^{\varepsilon}}{\partial x_{j}}\,sg^{\prime}_{n}\left(\frac{\partial u^{\varepsilon}}{\partial t}\right)\,\frac{\partial}{\partial x_{j}}\left(\frac{\partial u^{\varepsilon}}{\partial t}\right)\,dx\nonumber\\
  -\varepsilon\displaystyle\sum_{j=1}^{d}\int_{\Omega}B(u^{\varepsilon})\,\left(\frac{\partial}{\partial x_{j}}\left(\frac{\partial u^{\varepsilon}}{\partial t}\right)\right)^{2}\,sg^{\prime}_{n}\left(\frac{\partial u^{\varepsilon}}{\partial t}\right)\,dx.  \label{compactness12.eqn11}
\end{eqnarray}
We now prove that first two terms on the RHS of \eqref{compactness12.eqn11} tend to zero as $\varepsilon\to 0$. That is,
\begin{eqnarray}
\displaystyle\lim_{n\to\infty}\displaystyle\sum_{j=1}^{d}\int_{\Omega}f_{j}^{\prime}(u^{\varepsilon})\,\frac{\partial u^{\varepsilon}}{\partial t}\,sg^{'}_{n}\left(\frac{\partial u^{\varepsilon}}{\partial t}\right)\,\frac{\partial}{\partial x_{j}}\left(\frac{\partial u^{\varepsilon}}{\partial t}\right)\,dx =0,\label{compactness12.eqn6}\\
\displaystyle\lim_{n\to\infty}\varepsilon\displaystyle\sum_{j=1}^{d}\int_{\Omega}B^{'}(u^{\varepsilon})\,\frac{\partial u^{\varepsilon}}{\partial t}\,\frac{\partial u^{\varepsilon}}{\partial x_{j}}\,sg^{\prime}_{n}\left(\frac{\partial u^{\varepsilon}}{\partial t}\right)\,\frac{\partial}{\partial x_{j}}\left(\frac{\partial u^{\varepsilon}}{\partial t}\right)\,dx=0.\label{compactness12.eqn9}
\end{eqnarray}
\noindent{\bf Proof of \eqref{compactness12.eqn6}:}  
Since $\Big|\frac{\partial u^{\varepsilon}}{\partial t}\Big|\,sg^{'}_{n}\left(\frac{\partial u^{\varepsilon}}{\partial t}\right) < 1$, note that
\begin{eqnarray}\label{compactness12.eqn5B}
\left|\displaystyle\sum_{j=1}^{d}\int_{\left\{x\in\Omega\,:\,|\frac{\partial u^{\varepsilon}}{\partial t}|<\frac{1}{n}\right\}}\frac{\partial u^{\varepsilon}}{\partial t}sg^{'}_{n}\left(\frac{\partial u^{\varepsilon}}{\partial t}\right)\,\left(f_{1}^{\prime}(u^{\varepsilon}),\cdots,f_{d}^{\prime}(u^{\varepsilon})\right).\nabla\left(\frac{\partial u^{\varepsilon}}{\partial t}\right)\,dx\right|\hspace*{0.5cm}\nonumber\\
 \leq \sqrt{d}\,\displaystyle\max_{1\leq j\leq d}\left(\displaystyle\sup_{y\in I}\left|f^{'}_{j}(y)\right|\right)\,\int_{\left\{x\in\Omega\,:\,|\frac{\partial u^{\varepsilon}}{\partial t}|<\frac{1}{n}\right\}}\left|\nabla\left(\frac{\partial u^{\varepsilon}}{\partial t}\right) \right|\,dx.\hspace*{1cm}
\end{eqnarray}
 
Applying Lemma \ref{compactness12.lem1} with $v=\frac{\partial u^{\varepsilon}}{\partial t}$, the inequality \eqref{compactness12.eqn5B} gives \eqref{compactness12.eqn6}. 

\noindent{\bf Proof of \eqref{compactness12.eqn9}:}
Observe that
\begin{eqnarray}\label{compactness12.eqn8B}
 \left|\varepsilon\displaystyle\sum_{j=1}^{d}\int_{\Omega}B^{'}(u^{\varepsilon})\,\frac{\partial u^{\varepsilon}}{\partial t}\,\frac{\partial u^{\varepsilon}}{\partial x_{j}}\,sg^{\prime}_{n}\left(\frac{\partial u^{\varepsilon}}{\partial t}\right)\,\frac{\partial}{\partial x_{j}}\left(\frac{\partial u^{\varepsilon}}{\partial t}\right)\,dx \right|\hspace*{3cm}\nonumber\\ \leq \varepsilon\,\sqrt{d}\|B^{'}\|_{L^{\infty}(I)}\,\displaystyle\max_{1\leq j\leq d}\left(\left\|\frac{\partial u^{\varepsilon}}{\partial x_{j}}\right\|_{L^{\infty}(\Omega_{T})}\right) \int_{\left\{ x\in\Omega\,:\,\left|\frac{\partial u^{\varepsilon}}{\partial t}\right|<\frac{1}{n}\right\}}\left|\nabla\left(\frac{\partial u^{\varepsilon}}{\partial t}\right)\right|\,dx.
\end{eqnarray}

Applying Lemma \ref{compactness12.lem1} with $v=\frac{\partial u^{\varepsilon}}{\partial t}$, the inequality \eqref{compactness12.eqn8B} yields \eqref{compactness12.eqn9}. 

\noindent Since the third term on RHS of \eqref{compactness12.eqn13} is non-positive for every $\varepsilon>0$, on taking limit supremum on both sides of \eqref{compactness12.eqn11} yields
\begin{eqnarray}\label{compactness12.eqn13}
 \displaystyle\limsup_{n\to\infty}\,\int_{\Omega} u^\varepsilon_{tt}\,\,sg_{n}(u^\varepsilon_{t})\,dx &\leq& 0
\end{eqnarray}
in view of \eqref{compactness12.eqn6} and \eqref{compactness12.eqn9}. Note that the limit supremum in \eqref{compactness12.eqn13} is actually a limit, and as a consequence we get 
\begin{eqnarray}\label{compactness12.eqn14}
 \displaystyle\,\int_{\Omega} \frac{\partial}{\partial t}\left|u^\varepsilon_{t}\right|\,dx &\leq& 0
\end{eqnarray}
Integrating w.r.t. $t$ on both sides of \eqref{compactness12.eqn14}, and applying Fubini's theorem yields
\begin{eqnarray}\label{compactness12.eqn141}
 \int_{\Omega}\int_0^t \frac{\partial}{\partial t}\left|u^\varepsilon_{t}\right|\,d\tau\,dx &\leq& 0
\end{eqnarray}
Thus we get
\begin{eqnarray}\label{compactness12.eqn142}
 \int_{\Omega} \left(\left|u^\varepsilon_{t}(x,t)\right|-\left|u^\varepsilon_{t}(x,0)\right|\right)\,dx &\leq& 0
\end{eqnarray}
From the equation \eqref{regularized.IBVP}, we get 
\begin{eqnarray}\label{compactness12.eqn16}
 \frac{\partial u^{\varepsilon}}{\partial t}(x,0) &=& \varepsilon\displaystyle\sum_{j=1}^{d}\Big[B(u_{0\varepsilon})\,\frac{\partial^{2}u^{\varepsilon}}{\partial x_{j}^{2}}(x,0) + B^{'}(u_{0\varepsilon})\left(\frac{\partial u^{\varepsilon}}{\partial x_{j}}\right)^{2}(x,0)\Big]-\displaystyle\sum_{j=1}^{d} f_{j}^{'}(u_{0\varepsilon})\frac{\partial u^{\varepsilon}}{\partial x_{j}}(x,0)\nonumber\\
 &=&\varepsilon\displaystyle\sum_{j=1}^{d}\Big[B(u_{0\varepsilon})\,\frac{\partial^{2}u_{0\varepsilon}}{\partial x_{j}^{2}} + B^{'}(u_{0\varepsilon})\left(\frac{\partial u_{0\varepsilon}}{\partial x_{j}}\right)^{2}\Big]-\displaystyle\sum_{j=1}^{d} f_{j}^{'}(u_{0\varepsilon})\frac{\partial u_{0\varepsilon}}{\partial x_{j}}.
\end{eqnarray}
\textbf{Claim:} Let $u_{0}\in W^{2,1}(\Omega)\cup L^{\infty}_{c}(\Omega)$. Then for all $\varepsilon > 0$, there exists a constant
$C> 0$ such that 
\begin{eqnarray}\nonumber
 \int_{\Omega}\left(\frac{\partial u_{0\varepsilon}}{\partial x_{j}}\right)^{2}\,dx\leq C
\end{eqnarray}
Since $\mbox{supp}(u_{0})=K$ is a compact set, then we have $\mbox{supp}\left(\frac{\partial u_{0}}{\partial x_{j}}\right)\subset
\mbox{supp}(u_{0})=K$. Denote $R:=\mbox{dist}(K, \partial\Omega)$. Choose $0<\varepsilon <\frac{R}{4}$ and denote 
$$W:=\left\{x\in\Omega\,:\,0\leq\mbox{dist}(x,K)<\frac{R}{2}\right\}.$$
Then $K\subset W$. We consider  
\begin{eqnarray}\label{rsb21.eqn101}
 \int_{\Omega}\left(\frac{\partial u_{0\varepsilon}}{\partial x_{j}}\right)^{2}\,dx &=& \int_{\Omega}\left(\int_{\Omega}\frac{1}{\varepsilon^{d}}
 \rho\left(\frac{x-y}{\varepsilon}\right)\,\left(\frac{\partial u_{0}}{\partial x_{j}}(y)\,dy\right)\right)^{2}\,dx\nonumber\\
 &=&\int_{\Omega}\left(\int_{B(x,\varepsilon)}\frac{1}{\varepsilon^{d}}
 \rho\left(\frac{x-y}{\varepsilon}\right)\,\left(\frac{\partial u_{0}}{\partial x_{j}}(y)\,dy\right)\right)^{2}\,dx
\end{eqnarray}
Using the change of variable $\frac{x-y}{\varepsilon}=z$ in \eqref{rsb21.eqn101}, we have 
\begin{eqnarray}\label{rsb21.eqn102}
 \int_{\Omega}\left(\frac{\partial u_{0\varepsilon}}{\partial x_{j}}\right)^{2}\,dx &=& \int_{\Omega}\left(\int_{B(0,1)}(-1)^{d}
 \rho\left(z\right)\,\left(\frac{\partial u_{0}}{\partial x_{j}}(x-\varepsilon z)\,dz\right)\right)^{2}\,dx
\end{eqnarray}
Since $\left\{x-\varepsilon z\,\, :\,\, x\in K\,\mbox{and}\,\,z\in B(0,1)\right\}\subset W$, therefore from \eqref{rsb21.eqn102}, we have
\begin{eqnarray}\label{rsb21.eqn103}
  \int_{\Omega}\left(\frac{\partial u_{0\varepsilon}}{\partial x_{j}}\right)^{2}\,dx &=& \int_{W}\left(\int_{B(0,1)}(-1)^{d}
 \rho\left(z\right)\,\left(\frac{\partial u_{0}}{\partial x_{j}}(x-\varepsilon z)\,dz\right)\right)^{2}\,dx.
\end{eqnarray}
%Using \eqref{rsb21.eqn103} in \eqref{rsb21.eqn102}, we have
For $x\in W$, we now compute
\begin{eqnarray}\label{rsb21.eqn104}
 \frac{\partial u_{0\varepsilon}}{\partial x_{j}}(x) &=& \int_{B(x,\varepsilon)}\frac{1}{\varepsilon^{d}}
 \rho\left(\frac{x-y}{\varepsilon}\right)\,\left(\frac{\partial u_{0}}{\partial x_{j}}\right)(y)\,dy.
\end{eqnarray}
Using the change of variable $\frac{x-y}{\varepsilon}=z$ in \eqref{rsb21.eqn104}, we get
\begin{eqnarray}\label{rsb21.eqn105}
 \frac{\partial u_{0\varepsilon}}{\partial x_{j}}(x) &=& \int_{B(0,1)}
 (-1)^{d}\rho(z)\,\frac{\partial u_{0}}{\partial x_{j}}(x-\varepsilon z)\,dz.
\end{eqnarray}
Taking modulus on both sides of \eqref{rsb21.eqn105} and using H\"{o}lder inequality, we get
\begin{eqnarray}\label{rsb21.eqn106}
 \left|\frac{\partial u_{0\varepsilon}}{\partial x_{j}}(x)\right| &\leq& \int_{B(0,1)}\,\rho(z)\,\left|\left(\frac{\partial u_{0}}{\partial x_{j}}\right)(x-\varepsilon z
 )\right|\,dy,\nonumber\\
 &\leq& \left(\int_{B(0,1)}\rho(z)\,dz\right)^{2}\,\,\left(\int_{B(0,1)}\rho(z)\,\,\left|\frac{\partial u_{0}}{\partial x_{j}}(x-\varepsilon z)\right|^{2}
 \,\,dz\right)^{\frac{1}{2}}.
\end{eqnarray}
Applying Lemma \ref{regularized.lem1} in \eqref{compactness12.eqn16}, we get the existence of a constant $C_{1} > 0$, such that 
\begin{eqnarray}\label{rsb21.eqn107}
 \int_{\Omega}\left|\frac{\partial u^{\varepsilon}}{\partial t}\right|\,dx\,&\leq& C_{1}\|B\|_{L^{\infty}(I)} + \|B^{\prime}\|_{L^{\infty}(I)}
 \displaystyle\sum_{j=1}^{d}\left\|\frac{\partial u_{0}}{\partial x_{j}}\right\|_{L^{2}(\Omega)} + \displaystyle\max_{1\leq j\leq d}
\left(\|f_{j}^{\prime}\|_{L^{\infty}(I)}\right)\,\,TV_{\Omega}(u_{0}).\nonumber\\
{}
 \end{eqnarray}
Therefore we have
\begin{eqnarray}\label{rsb21.eqn108}
 \int_{\Omega_{T}}\left|\frac{\partial u^{\varepsilon}}{\partial t}\right|\,dx\,&\leq& T \left(C_{1}\|B\|_{L^{\infty}(I)} + \|B^{\prime}\|_{L^{\infty}(I)}
 \displaystyle\sum_{j=1}^{d}\left\|\frac{\partial u_{0}}{\partial x_{j}}\right\|_{L^{2}(\Omega)} + \displaystyle\max_{1\leq j\leq d}
\left(\|f_{j}^{\prime}\|_{L^{\infty}(I)}\right)\,\,TV_{\Omega}(u_{0})\right).\nonumber\\
{}
 \end{eqnarray}
 \textbf{Step 2:} Since $L^{1}(\Omega_{T})$ is compactly imbedded in the space of measures $M(\Omega_{T})$. Therefore we have
 \begin{eqnarray}\label{rsb21.eqn109}
  \left\|\frac{\partial u^{\varepsilon}}{\partial t}\right\|_{M(\Omega_{T})}&\leq& \left\|\frac{\partial u^{\varepsilon}}{\partial t}\right\|_{L^{1}(\Omega_{T})}
 \end{eqnarray}
In view of \eqref{rsb21.eqn108}, we see that $\left\{\frac{\partial u^{\varepsilon}}{\partial t}\right\}$ is bounded in the space of
measures $M(\Omega_{T})$.
The sequence $\left\{\frac{\partial u^{\varepsilon}}{\partial t}\right\}$ is bounded in $W^{-1,\infty}(\Omega_{T})$ as 
$\|u^{\varepsilon}\|_{L^{\infty}(\Omega_{T})}\leq \|u_{0}\|_{L^{\infty}(\Omega)}$. An application of Murat's Lemma \ref{chap9lem2},
 We get that the sequence $\left\{\frac{\partial u^{\varepsilon}}{\partial t}\right\}$ is compact in $H^{-1}(\Omega_{T})$.
\end{proof}

We need the following result for extraction of {\it a.e.} convergent subsequence of solutions $\left\{u^{\varepsilon}\right\}$  to generalized viscosity problem
\begin{theorem}\label{compactness12.thm2}
 {\rm Assume \textbf{Hypothesis D} and let $\left\{u^{\varepsilon}\right\}$ be the sequence of solutions to 
 viscosity problem \eqref{regularized.IBVP} and $d=2$. Then there exists a subsequence $\left\{u^{\varepsilon_{k}}\right\}$ of $\left\{u^{\varepsilon}\right\}$ and a function $u$ such that 
 \begin{equation}\label{compactness12.eqn19}
  u^{\varepsilon_{k}}\to u\,\,\mbox{as}\,k\to\infty.
 \end{equation}
 }
\end{theorem}
\begin{proof}
 Multiplying the generalized viscosity problem by $f_{1}^{'}$, $f_{2}^{'}$, we get
\begin{eqnarray}\label{compactness12.eqn20}
 \left(f_{1}^{'}(u^{\varepsilon})\right)^{2}\,\frac{\partial u^{\varepsilon}}{\partial x_{1}} + f_{1}^{'}(u^{\varepsilon})f_{2}^{'}(u^{\varepsilon})\,\frac{\partial u^{\varepsilon}}{\partial x_{2}}&=& \varepsilon\displaystyle\sum_{j=1}^{2}f_{1}^{'}(u^{\varepsilon})\frac{\partial}{\partial x_{j}}\left(B(u^{\varepsilon})\,\frac{\partial u^{\varepsilon}}{\partial x_{j}}\right)-\frac{\partial f_{1}(u^{\varepsilon})}{\partial t},\nonumber\\
 f_{1}^{'}(u^{\varepsilon})f_{2}^{'}(u^{\varepsilon})\,\frac{\partial u^{\varepsilon}}{\partial x_{1}} + \left(f_{2}^{'}(u^{\varepsilon})\right)^{2}\,\frac{\partial u^{\varepsilon}}{\partial x_{2}} &=& \varepsilon\displaystyle\sum_{j=1}^{2}f_{2}^{'}(u^{\varepsilon})\frac{\partial}{\partial x_{j}}\left(B(u^{\varepsilon})\,\frac{\partial u^{\varepsilon}}{\partial x_{j}}\right)-\frac{\partial f_{2}(u^{\varepsilon})}{\partial t}.
\end{eqnarray}
Denote
\begin{eqnarray}\label{compactness12.eqn21}
 F_{11}(\lambda) :=\int_{0}^{\lambda}\left(f_{1}^{'}(s)\right)^{2}\,ds,\nonumber\\
 F_{12}(\lambda) :=\int_{0}^{\lambda}f_{1}^{'}(s)f_{2}^{'}(s)\,ds,\nonumber\\
 F_{22}(\lambda) :=\int_{0}^{\lambda}\left(f_{2}^{'}(s)\right)^{2}\,ds.
\end{eqnarray}
Equation \eqref{compactness12.eqn20} can be rewritten as 
\begin{eqnarray}\label{compactness12.eqn22}
 \frac{\partial F_{11}(u^{\varepsilon})}{\partial x_{1}} + \frac{\partial F_{12}(u^{\varepsilon})}{\partial x_{2}} &=& \varepsilon\displaystyle\sum_{j=1}^{2}\frac{\partial}{\partial x_{j}}\left(B(u^{\varepsilon})\,\frac{\partial f_{1} (u^{\varepsilon})}{\partial x_{j}}\right)-\varepsilon\displaystyle\sum_{j=1}^{2}B(u^{\varepsilon})\left(\frac{\partial u^{\varepsilon}}{\partial x_{j}}\right)^{2}f_{1}^{''}(u^{\varepsilon})-\frac{\partial f_{1}(u^{\varepsilon})}{\partial t},\nonumber\\
 \frac{\partial F_{12}(u^{\varepsilon})}{\partial x_{1}} + \frac{\partial F_{22}(u^{\varepsilon})}{\partial x_{2}} &=& \varepsilon\displaystyle\sum_{j=1}^{2}\frac{\partial}{\partial x_{j}}\left(B(u^{\varepsilon})\,\frac{\partial f_{2}(u^{\varepsilon}) }{\partial x_{j}}\right)-\varepsilon\displaystyle\sum_{j=1}^{2}B(u^{\varepsilon})\left(\frac{\partial u^{\varepsilon}}{\partial x_{j}}\right)^{2}f_{2}^{''}(u^{\varepsilon})-\frac{\partial f_{2}(u^{\varepsilon})}{\partial t}.\nonumber\\
 &
\end{eqnarray}
We now show that RHS of two equations of \eqref{compactness12.eqn22} lie in a compact set of $H^{-1}(\Omega_{T})$. For that we show that for $i=1,2$,
\begin{enumerate}
 \item[(i).]
 \begin{equation}\label{compactness12.eqn23}
  \varepsilon\displaystyle\sum_{j=1}^{2}\frac{\partial}{\partial x_{j}}\left(B(u^{\varepsilon})\,\frac{\partial f_{i} (u^{\varepsilon})}{\partial x_{j}}\right)\to 0\,\,\mbox{in}\,\,H^{-1}(\Omega_{T}),
 \end{equation}
\item[(ii).] 
\begin{equation}\label{compactness12.eqn24}
 -\varepsilon\displaystyle\sum_{j=1}^{2}B(u^{\varepsilon})\left(\frac{\partial u^{\varepsilon}}{\partial x_{j}}\right)^{2}f_{1}^{''}(u^{\varepsilon})-\frac{\partial f_{1}(u^{\varepsilon})}{\partial t}\,\,\mbox{is bounded in the space of measure}\,\,M(\Omega_{T}).
\end{equation}
\end{enumerate}
Firstly, we prove \eqref{compactness12.eqn23}. Observe that 
\begin{eqnarray}\label{compactness12.eqn25}
 \Big\|\varepsilon\displaystyle\sum_{j=1}^{2}\frac{\partial}{\partial x_{j}}\left(B(u^{\varepsilon})\,\frac{\partial f_{i} (u^{\varepsilon})}{\partial x_{j}}\right)\Big\|_{H^{-1}(\Omega_{T})} &=& \sup\Big\{\Big|\int_{0}^{T}\int_{\Omega}\left(\varepsilon\displaystyle\sum_{j=1}^{2}\frac{\partial}{\partial x_{j}}\left(B(u^{\varepsilon})\,\frac{\partial f_{i} (u^{\varepsilon})}{\partial x_{j}}\right)\right)\,\nonumber\\ &&\phi(x,t)\,dx\,dt\Big|\,; \|\phi\|_{H^{1}_{0}(\Omega_{T})}\leq 1\Big\},\nonumber\\
 &=& \sup\Big\{\Big|-\int_{0}^{T}\int_{\Omega}\left(\varepsilon\displaystyle\sum_{j=1}^{2}\left(B(u^{\varepsilon})\,f^{'}_{i} (u^{\varepsilon})\frac{\partial u^{\varepsilon}}{\partial x_{j}}\right)\right)\,\nonumber\\ &&\frac{\partial\phi}{\partial x_{j}}(x,t)\,dx\,dt\Big|\,; \|\phi\|_{H^{1}_{0}(\Omega_{T})}\leq 1\Big\},\nonumber\\
 &\leq& \varepsilon \|B\|_{L^{\infty}(I)}\,\left(\displaystyle\max_{1\leq i\leq 2}\|f_{i}^{'}\|_{L^{\infty}(I)}\right)\,\|\nabla u^{\varepsilon}\|_{\left(L^{2}(\Omega_{T})\right)^{2}}.\nonumber\\
 &
\end{eqnarray}
Since $\sqrt{\varepsilon}\|\nabla u^{\varepsilon}\|_{\left(L^{2}(\Omega_{T})\right)^{2}}\leq C,$ which is independent of $\varepsilon$, therefore we have \eqref{compactness12.eqn23} and $\varepsilon\displaystyle\sum_{j=1}^{2}\frac{\partial}{\partial x_{j}}\left(B(u^{\varepsilon})\,\frac{\partial f_{i} (u^{\varepsilon})}{\partial x_{j}}\right)$ lie in a compact set of $H^{-1}(\Omega_{T})$.

Next we show \eqref{compactness12.eqn24}. We know that $L^{1}(\Omega_{T})$ is continuously imbeeded in $\left(L^{\infty}(\Omega_{T})\right)^{\ast}$. Therefore we have 
\begin{eqnarray}\label{compactness12.eqn26}
 \Big\| -\varepsilon\displaystyle\sum_{j=1}^{2}B(u^{\varepsilon})\left(\frac{\partial u^{\varepsilon}}{\partial x_{j}}\right)^{2}f_{1}^{''}(u^{\varepsilon})-\frac{\partial f_{1}(u^{\varepsilon})}{\partial t}\Big\|_{M(\Omega_{T})} &\leq& \int_{0}^{T}\int_{\Omega}\Big|\varepsilon\displaystyle\sum_{j=1}^{2}B(u^{\varepsilon})\left(\frac{\partial u^{\varepsilon}}{\partial x_{j}}\right)^{2}f_{1}^{''}(u^{\varepsilon})\nonumber\\
 &&-f_{1}^{'}(u^{\varepsilon})\frac{\partial u^{\varepsilon}}{\partial t}\Big|\,dx\,dt,\nonumber\\
 &\leq& \|B\|_{L^{\infty}(I)}\,\displaystyle\max_{1\leq i\leq 2}\left(\displaystyle\sup_{y\in I}\Big|f_{i}^{''}(y)\Big|\right)\left(\sqrt{\varepsilon}\,\|\nabla u^{\varepsilon}\|_{\left(L^{2}(\Omega_{T})\right)^{2}}\right)^{2} \nonumber\\
 &&+ \displaystyle\max_{1\leq i\leq 2}\left(\displaystyle\sup_{y\in I}\Big|f_{i}^{''}(y)\Big|\right)\,\Big\|\frac{\partial u^{\varepsilon}}{\partial t}\Big\|_{L^{1}(\Omega_{T})}.
\end{eqnarray}
Applying Theorem \ref{timedervative.thm1} and $\sqrt{\varepsilon}\|\nabla u^{\varepsilon}\|_{\left(L^{2}(\Omega_{T})\right)^{2}}\leq C,$ which is independent of $\varepsilon$, we get \eqref{compactness12.eqn24}.\\
We want to use Theorem of compensated compactness \cite[p.31]{Dacorogna} to conclude the almost every convergence of 
$\left\{u^{\varepsilon}\right\}$ to a function $u$ in $L^{\infty}(\Omega_{T})$.
Observe that 
\begin{eqnarray}\label{Comp.Theorem.eqn1}
 F_{11}(u^{\varepsilon})&\rightharpoonup& \overline{F}_{11}\,\,\mbox{in}\,\,\, L^{2}(\Omega_{T})\,\,\mbox{as}\,\,\varepsilon\to 0,\nonumber\\
 F_{12}(u^{\varepsilon})&\rightharpoonup& \overline{F}_{12}\,\,\mbox{in}\,\,\, L^{2}(\Omega_{T})\mbox{as}\,\,\varepsilon\to 0\,\,,\nonumber\\
 F_{22}(u^{\varepsilon})&\rightharpoonup& \overline{F}_{22}\,\,\mbox{in}\,\,\, L^{2}(\Omega_{T})\mbox{as}\,\,\varepsilon\to 0\,\,.
\end{eqnarray}
Therefore we have
$$\left(F_{11}(u^{\varepsilon}), F_{12}(u^{\varepsilon}), F_{12}(u^{\varepsilon}), F_{22}(u^{\varepsilon}) \right)\rightharpoonup \left(\overline{F}_{11},\overline{F}_{12},\overline{F}_{12}, \overline{F}_{22} \right)\,\,\mbox{as}\,\,\varepsilon\to 0.$$
The following combinations 
\begin{eqnarray}\label{comp.Theorem.eqn2}
 \frac{\partial}{\partial x_{1}}F_{11}(u^{\varepsilon}) + \frac{\partial}{\partial x_{2}}F_{12}(u^{\varepsilon})\,\, &,& \frac{\partial}{\partial x_{1}}F_{12}(u^{\varepsilon}) + \frac{\partial}{\partial x_{2}}F_{22}(u^{\varepsilon}),\nonumber\\
 \frac{\partial}{\partial t}F_{11}(u^{\varepsilon}) &,& \frac{\partial}{\partial t}F_{22}(u^{\varepsilon}),
\end{eqnarray}
are compact in $H^{-1}(\Omega_{T})$.\\
Consider the set 
$$\nu :=\left\{\left(\lambda,\xi\right)\in \mathbb{R}^{4}\times\mathbb{R}^{3}\setminus\left\{0\right\}\,\,;\,\,\lambda_{1}\xi_{1}+\lambda_{2}\xi_{2}=0,\,\,
\lambda_{1}\xi_{1}+\lambda_{2}\xi_{2}=0,\,\,\lambda_{1}\xi_{0}=0,\,\,\lambda_{3}\xi_{0}=0\right\}.$$
The quadratics $Q(F_{11}(u^{\varepsilon}), F_{12}(u^{\varepsilon}), F_{12}(u^{\varepsilon}), F_{22}(u^{\varepsilon}))$ which vanish 
on the projections,
$$\Lambda=\left\{\lambda\in\mathbb{R}^{4}\,\,:\,\,\left(\lambda,\xi\right)\in\nu\right\},$$
which is precisely
$$\left\{\lambda\in\mathbb{R}^{4}\,\,:\,\,\lambda_{1}\lambda_{4}-\lambda_{2}\lambda_{3}=0\right\}.$$
Therefore we have 
\begin{eqnarray}\label{comp.Theorem.eqn3}
 \left(F_{11}(u^{\varepsilon}),\,F_{12}(u^{\varepsilon})\right)\cdot\left(F_{22}(u^{\varepsilon}),\,F_{12}(u^{\varepsilon})\right)\rightharpoonup \overline{F}_{11}\overline{F}_{22}
 -\overline{F}_{12}^{2}\,\,\mbox{in}\,\,L^{2}(\Omega_{T})\,\,\mbox{as}\,\,\varepsilon\to 0.
\end{eqnarray}
If we express \eqref{comp.Theorem.eqn3} in term of youngs measures $\nu_{x,t}(\cdot)$, we write
\begin{eqnarray}\label{comp.Theorem.eqn4}
 \Big\langle \nu_{x,t}, \left(F_{11}(\lambda)-\overline{F}_{11}\right)\cdot\left(F_{22}(\lambda)-\overline{F}_{22}\right)-\left(F_{12}(\lambda)-\overline{F}_{12}\right)^{2}\Big\rangle=0.
\end{eqnarray}
We now want to show that a subsequence of $\left\{u^{\varepsilon}\right\}$ converges {\it a.e.} to a function $u$ in $L^{\infty}(\Omega_{T})$.
We repeat the following proof from \cite[p.702-p.703]{Tadmor}. Let us consider the following nonnegative function
$$D(w):= \left(F_{11}(w)-F_{11}(c)\right)\cdot\left(F_{22}(w)-F_{22}(c)\right)-\left(F_{12}(w)-F_{12}(c)\right)^{2},$$
where $c=c(x,t)$ denotes an arbitary function which needs to be determined and is independent of $u^{\varepsilon}$. \\ 
Using Cauchy-Swartz inequality, we observe that 
\begin{eqnarray}\label{comp.Theorem.eqn5}
 \left(F_{12}(w)-F_{12}(c)\right)^{2} \leq \left(F_{11}(w)-F_{11}(c)\right)\cdot\left(F_{22}(w)-F_{22}(c)\right).
\end{eqnarray}
Therefore $D(w)$ is nonnegative. Using \eqref{comp.Theorem.eqn3}, we get 
\begin{eqnarray}\label{comp.Theorem.eqn6}
 D(u^{\varepsilon})&=& \left(F_{11}(u^{\varepsilon})-F_{11}(c)\right)\cdot\left(F_{22}(u^{\varepsilon})-F_{22}(c)\right) - \left(F_{12}(u^{\varepsilon})-F_{12}(c)\right)^{2},\nonumber\\
&=& \left[\left(F_{11}(u^{\varepsilon})-\overline{F}_{11}\right) + \left(\overline{F}_{11}-F_{11}(c)\right)\right]\cdot
\left[\left(F_{22}(u^{\varepsilon})-\overline{F}_{22}\right)+ \left(\overline{F}_{22}-F_{22}(c)\right)\right]\nonumber\\
&-& \Big[\left(F_{12}(u^{\varepsilon})-\overline{F}_{12}\right)^{2} + 2\,\left(F_{12}(u^{\varepsilon})-\overline{F}_{12}\right)
\left(\overline{F}_{12}-F_{12}(c)\right) + \left(\overline{F}_{12}-F_{12}(c)\right)^{2}\Big],\nonumber\\
&\rightharpoonup& \left[\left(\overline{F}_{11}-F_{11}(c)\right)\cdot\left(\overline{F}_{22}-F_{22}(c)\right)-\left(\overline{F}_{12}-F_{12}(c)\right)\right].
\end{eqnarray}
Since $0\leq \overline{F}_{11}\leq \int_{\mbox{min\,u}}^{\mbox{max\,u}}\left(f_{1}^{\prime}(v)\right)^{2}\,dv$, there exists $c=c(x,t)$ such that
$\int^{c}\left(f_{1}^{\prime}(v)\right)^{2}\,dv=\overline{F}_{11}$. For this $c=c(x,t)$, we have $F_{11}(c)-\overline{F}_{11}=0$. From equations 
\eqref{comp.Theorem.eqn5} and \eqref{comp.Theorem.eqn6}, we have $D(u^{\varepsilon})\rightharpoonup 0$. Since $D(u^{\varepsilon})$ is a bounded function, we have 
$D^{2}(u^{\varepsilon})\rightharpoonup 0$. Therefore $D(u^{\varepsilon})\to 0$ storngly.\\ It is easy to see that $D(w)$ has minimum at$w=c$. Next assume that
$f_{1}^{\prime}$ and $f_{2}^{\prime}$ are linearly independent, {\it i.e.,}
$$\forall |\xi|=1,\,\,\, S(\xi, .):=\xi_{1}f_{1}^{\prime}(v) +\xi_{2}f_{2}f_{2}^{\prime}(v)\neq 0$$
on any nontrivial interval. Then the Cauchy-Swartz inequality in \eqref{comp.Theorem.eqn5} is strict. Therefore we have $D(c)=0$
is the strict minimun of $D(w)$, {\it i.e.,} $D(w)> D(c)$ for $w\neq c$. Then strong convergence of $D(u^{\varepsilon})\to 0$ implies
strong convergence of a subsequence of $u^{\varepsilon}$ to  $c(x,t)= u(x,t)$.
\end{proof}
\\

\textbf{Proof of Theorem \ref{paper2.compensatedcompactness.theorem2}:}
 The proof of Theorem \eqref{paper2.compensatedcompactness.theorem2} follows from the proof of Theorem 1.1 of \cite[p.31]{Ramesh}.
\\
\newpage
%\bibliography{Ramesh-bibfile}
\bibliographystyle{amsplain}

%\end{document}
\end{document}